\documentclass[11pt]{article}
\usepackage{tgpagella}
\usepackage[utf8]{inputenc}
\usepackage[T1]{fontenc}
\usepackage{amssymb}
\usepackage{amsfonts}
\usepackage{amsmath}
\usepackage{amsthm}
\usepackage{mathrsfs}
\usepackage{comment}
\usepackage{hyperref}
\newtheorem{theorem}{Theorem}
\newtheorem{proposition}[theorem]{Proposition}
\newtheorem{lemma}[theorem]{Lemma}

\newtheorem{corollary}[theorem]{Corollary}

\theoremstyle{definition}
\newtheorem{definition}[theorem]{Definition}
\newtheorem{convention}[theorem]{Convention}
\newtheorem{example}[theorem]{Example}
\newtheorem{remark}[theorem]{Remark}
\newtheorem{question}[theorem]{Question}

\newcommand{\PA}{\textnormal{PA}}

\newcommand{\set}[2]{\lbrace #1 \ \mid \ #2 \rbrace }

\newcommand{\CT}{\textnormal{CT}}
\newcommand{\df}[1]{\textbf{#1}}
\newcommand{\num}[1]{\underline{#1}}

\newcommand{\ElDiag}{\textnormal{ElDiag}}
\newcommand{\val}[1]{{#1}^{\circ}}
\newcommand{\Val}{\textnormal{Val}}
\newcommand{\tuple}[1]{\langle #1 \rangle}

\newcommand{\dpt}{\textnormal{dp}}
\newcommand{\Th}{\textnormal{Th}}
\newcommand{\Lang}{\mathscr{L}}
\newcommand{\dom}{\textnormal{dom}}
\newcommand{\Con}{\textnormal{Con}}

\newcommand{\FV}{\textnormal{FV}}
\newcommand{\qcr}[1]{\ulcorner #1 \urcorner}
\newcommand{\rk}{\textnormal{rk}}
\newcommand{\initialeq}{\unlhd}
\newcommand{\initial}{\lhd}

\newcommand{\LPA}{\mathscr{L}_{\PA}}
\newcommand{\form}{\textnormal{Form}}
\newcommand{\Term}{\textnormal{Term}}
\newcommand{\Sent}{\textnormal{Sent}}

\newcommand{\ClTerm}{\textnormal{ClTerm}}

\newcommand{\ClTermSeq}{\textnormal{ClTermSeq}}
\newcommand{\VarSeq}{\textnormal{VarSeq}}
\newcommand{\Comp}{\textnormal{Comp}}
\newcommand{\Asn}{\textnormal{Asn}}
\newcommand{\Var}{\textnormal{Var}}
\newcommand{\Num}{\textnormal{Num}}
\newcommand{\SubAsn}{\textnormal{SubAsn}}

\newcommand{\Aut}{\textnormal{Aut}}
\newcommand{\tp}{\textnormal{tp}}

\title{Full satisfaction classes, definability, and automorphisms} 
\author{Bartosz Wcisło\footnote{Institute of Mathematics, Polish Academy of Sciences.}}

\begin{document}
	
	\maketitle

\begin{abstract}
	 We show that for every countable recursively saturated model $M$ of Peano Arithmetic and every subset $A \subseteq M$, there exists a full satisfaction class $S_A \subset M^2$ such that $A$ is definable  in $(M,S_A)$ without parametres. It follows that in every such model, there exists a full satisfaction class which makes every element definable and thus the expanded model is minimal and rigid. On the other hand, we show that for every full satisfaction class $S$ there are two elements which have the same arithmetical type, but exactly one of them is in $S$. In particular, the automorphism group of a model expanded with a satisfaction class is never equal to the automorphism group of the original model. The analogue of many of the results proved here for full satisfaction classes were obtained by Roman Kossak for partial inductive satisfaction	classes. However, most of the proofs relied heavily on the induction scheme in a crucial way, so recapturing the results in the setting of full satisfaction classes requires quite different arguments.
\end{abstract}	
	
\section{Introduction}

Satisfaction classes are a classical topic in the study of models of Peano Arithmetic. A full satisfaction class is essentially a function which assigns to each formula in the sense of a model a set of elements which "satisfy" that formula. We require that this assignment is compositional in that, for instance, the set of elements satisfying a conjunction $\phi \wedge \psi$ is the intersection of the sets of elements satisfying $\phi$ and $\psi$. In that way we endow nonstandard arithmetical formulae with certain "semantics." We can also weaken the requirement that the relation of satisfaction is defined for all formulae, thus obtaining a notion of a partial satisfaction class, or strengthen it in a number of interesting ways. For instance, we may require that the expansion of a model $M$ with a satisfaction class $S$ satisfies full induction, obtaining a notion of an inductive satisfaction class. 

A systematic study of satisfaction classes was initiated by Krajewski,
%%%%%%%% TODO: GDZIE
 preceded by a study of "nonstandard semantics" by Robinson, \cite{robinson}. Probably the most important result of the whole area was obtained by Kotlarski, Krajewski, and Lachlan in \cite{kkl}, where it was proved that in every countable recursively saturated model of $\PA$ there exists a full satisfaction class. The saturation assumption turned out to be essential, since it was subsequently shown by Lachlan in \cite{lachlan} that indeed every model with a full satisfaction class is recursively saturated.  Later, it was shown that the presence of a full satisfaction class in a model is a strictly stronger requirement than recursive saturation. In \cite{smith} it was proved by Smith that every model $M \models \PA$ with a full satisfaction class has an undefinable subset whose all bounded initial segments are coded in the model. On the other hand, by the previous work of Kaufmann, \cite{kaufmann}, we know that there are uncountable recursively saturated models in which no such subset can be found (so called rather classless models). 

Model-theoretic properties of the satisfaction classes themselves were also a subject of study. In \cite{kossak_note}, a number of results typical for models of $\PA$ were proved by Kossak for the models expanded with satisfaction classes. For instance, it was shown that in every countable recursively saturated model $M$ of $\PA$, there exists a class $S \subset M$ such that $(M,S)$ is minimal, i.e., it does not have a proper elementary submodel. It was also proved that in any countable model $M \models \PA$, there exists a family of continuum many pairwise non-elementary-equivalent satisfaction classes. The work of Kossak was later used by him in a joint paper with Schmerl, \cite{kossakschmerl_satisfaction}, to obtain new results on models of pure $\PA$. 

The findings from \cite{kossak_note} mostly applied to inductive satisfaction classes, so most of them do not typically generalise when we require that the constructed class is also full, since the presence of a full inductive satisfaction class has nontrivial implications for the theory of the underlying model (e.g., it can be verified that such a model satisfies $\Con_{\PA}$). The proofs in \cite{kossak_note} mostly relied on the induction scheme in a crucial way, so recapturing the results in the setting of full satisfaction classes requires quite different arguments.\footnote{This paper actually answers some questions of Roman Kossak, who asked during his visit in Warsaw whether certain results from that previous work also hold for full satisfaction classes.}

In this article, we investigate such full satisfaction classes with special model-theoretic properties. Among other things, we show that in a countable model of $\PA$, every set is definable (without parametres) from a full satisfaction class. This has a number of model-theoretic corollaries. For instance, we can require that every element in the model is definable from $S$ without parametres, so that all elements of $M$ become distinguishable in the presence of a full satisfaction class. On the other hand, we show that if $S$ is a full satisfaction class on $M$, we can always find a pair of sentences $\phi, \psi \in \Sent_{\LPA}(M)$ with the same arithmetical types such that $\phi$ is rendered true and $\psi$ is rendered false, so it cannot happen that all elements indistinguishable by arithmetical formulae are also indistinguishable from the perspective of a given satisfaction class.

The study of satisfaction classes closely parallels the study of truth predicates which is an active research area in philosophical logic.\footnote{A classical textbook on axiomatic truth predicates is \cite{halbach}. See also \cite{cies_ksiazka} for an overview of some more recent results in that field.} Although the results in both terminological traditions seem to go hand-in-hand, they are typically not formulated in a way that allows for a direct translation. In this article, we try to close this small gap and propose certain additional technical assumptions both on satisfaction and on truth classes that makes the correspondence completely direct. In particular, all our results apply, \textit{mutatis mutandis}, both to satisfaction classes and to truth predicates.

\section{Preliminaries}

\subsection{Arithmetic}
In this article, we will consider satisfaction classes over Peano Arithmetic, $\PA$. We will assume that the reader has some previous knowledge of $\PA$ and its models. Crucially for our purposes, Peano Arithmetic can define  Ackermann's elementhood notion $x \in_{\textnormal{Ack}} y$ defined as "$x$-th bit of the binary expansion of $y$ is one." This notion, provably in $\PA$, satisfies the axioms of the Zermelo--Fr\"ankel set theory, except for the Axiom of Infinity. This allows us to treat $\PA$ as a theory of finite sets (consequently, we will denote the relation $\in_{\textnormal{Ack}}$ simply with "$\in$"). In particular, $\PA$ is capable of handling syntactic notions such as formulae, sentences of proofs. All the  relevant background facts including formalisation of syntax can be found e.g. in \cite{kaye}, especially Chapters 1--9 or \cite{hajekpudlak}, especially Chapter I. 

Let us introduce some notation for the formalised syntactic objects. 
\begin{definition} \label{def_formalised_syntax} \
	Let $\Lang$ be a primitive recursive langauge containing $\LPA$ (i.e, $\Lang$ is countable and the function which given a G\"odel number returns the type of symbol it represents and its arity is primitive recursive).
	\begin{itemize}
		\item The formula $\Var(x)$ expresses that $x$ is (a code of) a first-oder variable.
		\item The formula $\VarSeq(x)$ expresses that $x$ is (a code of) a sequence of first-order variables. 
		\item The formula $\Term_{\Lang}(x)$ expresses that $x$ is (a code of) a term over the language $\Lang$.
		\item The formula $\ClTerm_{\Lang}(x)$ expresses that $x$ is (a code of) a closed term over the language $\Lang$.
		\item The formula $\ClTermSeq_{\Lang}(x)$ expresses that $x$ is (a code of) a sequence of closed terms over the language $\Lang$.
		\item The formula $\form_{\Lang}(x)$ expresses that $x$ is (a code of) a formula over the language $\Lang$.
		\item The formula $\Sent_{\Lang}(x)$ expresses that $x$ is (a code of) a sentence over the language $\Lang$.
		\item The formula $\FV_{\Lang}(x,y)$ expresses that $y$ is (a code of) a set of free variables of $x$, where $x$ is a term or a formula in $\Lang$. In what follows, we will omit the subscript $\Lang$ and assume that we have one formula $\FV$ which works uniformly for all the relevant languages.
		\item The formula $\Asn_{\Lang}(x,y)$ expresses that $y$ is an $x$-assignment, i.e. $x$ is (a code of) a term or a formula and $y$ is a function whose domain contains $\FV_{\Lang}(x)$. As in the case of $\FV$, we will omit the mention of the language in the formula $\Asn$. 
		\item If $\alpha, \beta$ are two assignments and $\bar{v} = \tuple{v_1,\ldots,v_c} \in \VarSeq$, by $\beta \sim_{\bar{v}} \alpha$, we mean that $\dom \beta \supseteq \dom \alpha \cup \{v_1, \ldots, v_c\}$ and $\beta(v) = \alpha(v)$ for every $v \in \dom(\alpha) \setminus \{v_1, \ldots, v_c\}.$ 
		If $\bar{v}$ is a sequence of length one of the form $\tuple{v}$, we will write $\beta \sim_{v} \alpha$.
		\item If $t \in \Term_{\LPA}$ and $\alpha \in \Asn(t)$, the formula $\Val(t,\alpha,x)$ expresses that $x$ is a formally computed value of $t$ under the assignment $\alpha$. For instance, if $t = \qcr{0 + v \times (S0 + S0)}$ and $\alpha$ is an assignment which sends the variable $v$ to $5$, then $\Val(t,\alpha,10)$ holds. We will also denote it with $t^{\alpha} =x$. In the special case when $t$ is a closed term and $\alpha$ is the empty valuation, we will use the notation $t^{\circ} = x$.
		\item By $\Num(x,y)$ we mean that $y$ is (a code of) the formal numeral $S\ldots S0$ of $\LPA$, where $S$ is iterated $x$ times. We will also denote it with $\num{x} = y$.  
		\item If $\phi \in \form_{\Lang}$ and $\alpha \in \Asn(\phi)$, then the formula $\SubAsn(\phi,\alpha,x)$  expresses that $x$ is (a code of) the sentence resulting by substituting the numeral $\num{\alpha(v)}$ for each free variable $v$ in $\phi$. We will also denote it with $\phi[\alpha] = x$. 
		\item If $\phi \in \form_{\Lang}$, then the formula $\dpt(\phi,x)$ expresses that the syntactic depth of $\phi$ is equal to $x$, where the syntactic depth is the depth of nesting of quantifiers and connectives in $\phi$ (thus counting the depth of all atomic formulae, possibly featuring very complex terms, as $0$). We will also denote it with $\dpt(\phi) = x$. 
	\end{itemize}
\end{definition}

In the article, we will make use of a number of notational conventions.

\begin{convention} \label{conv_konwencje_syntaktyczne} \
	\begin{itemize}
		\item We will often suppress formal notation referring to syntactic operations and write the results of these operations instead. For instance, we will write $S(\phi \wedge \psi,\alpha)$ rather than "For all $z$, if $z$ is the unique formula which is a conjunction of $\phi$ and $\psi$, then $S(z,\alpha)$."
		\item We will use the notions defined with functional expressions such as $\num{x}, \val{t}, \FV(\phi)$ as if they had a stand-alone meaning (as if they were terms). 
		\item We will use the expressions such as $\form_{\LPA}$ or $\Asn$ as if they were denoting sets, e.g., writing $\phi \in \form_{\LPA}$ or $\alpha \in \Asn(\phi)$ (note that in the latter case, we also make use of the previous convention).
		\item If $M \models \PA, \phi \in \form_{\LPA}(M), \bar{v} \in \VarSeq(M), Q \in \{\forall, \exists\}$, then by $Q \bar{v} \phi$,  we mean $Q v_1 \ldots Q v_c \phi$, where $\bar{v} = \tuple{v_1, \ldots, v_c}$.
		%%%%%%%%%%%%%%% TODO czy to jest konwencja?
	\end{itemize}
\end{convention}

Let us recall some basic definitions and facts about models of $\PA$. 

\begin{definition} \label{def_recursive_saturation_resplendence}
	We say that a model $M$ is \df{recursively saturated} if every recursive type with finitely many parametres and variables is satisfied.\footnote{Recursive type is the same thing as a computable type, i.e., a consistent set of formulae which is computable as a set of G\"odel codes. In this article, we stick to the "recursive" terminology, since it is prevalent in the area of models of arithmetic.} We say that a model $M$ is \df{resplendent} if for every recursive theory $\Th$ consistent with $\ElDiag(M)$, $M$ can be expanded to a model of $\Th$.
\end{definition}

The notion of resplendence is clearly very rich in consequences. Thanks to its link with the recursive saturation, it is also rather easy to meet in nature (see \cite{kaye}, Theorems 15.7 and 15.8):

\begin{theorem}[Barwise--Schlipf, Ressayre] \label{th_recursive_saturated_resplendent}
	Every countable recursively saturated model in a countable recursive language is resplendent. Moreover, if $M$ is a countable recursively saturated model and $\Th$ is a countable recursive theory, then we can find an expansion of $M$ to a model of $\Th$ which is itself recursively saturated.
\end{theorem}

Another important feature of recursively saturated models is that they allow very powerful back-and-forth constructions which are very rich in structural consequences.
\begin{definition}
	Let $M \models \PA$. By the \df{standard system} of $M$, we mean the family of $X \subseteq \omega$ such that there exists $c \in M$ for which $X = \set{n \in \omega}{M \models n \in_{\textnormal{Ack}} c}.$
\end{definition}

\begin{theorem}[Wilmers] \label{th_classification}
	Suppose that $M, N \models \PA$ are recursively saturated models. Then $M \simeq N$ iff $M \equiv N$ and $M,N$ have the same standard systems.
\end{theorem}
The result was first obtained by Wilmers in his PhD thesis \cite{wilmers}. The proof may be found in \cite{kaye}, Theorem 15.2.3, where it is stated in  greater generality (for so called rich theories). For us, the following corollary will be relevant.\footnote{Strictly speaking, the corollary does not follow from Theorem \ref{th_classification} as stated. However, it readily follows from its proof or from  a more general statement, like Theorem 15.2.3 from \cite{kaye}.}
\begin{theorem} \label{th_recursive_saturated_automorphism}
	Let $M \models \PA$ be a recursively saturated model. Suppose that $a,b \in M$ have the same types. Then there exists an automorphism $f \in \Aut(M)$ such that $f(a) =b$.
\end{theorem}

\subsection{Satisfaction classes}

In the literature, there are two concurrent terminological traditions. One of them speaks of satisfaction classes and is more widespread in the study of models of arithmetic. The other tradition, which stems from philosophical logic, tends to formulate results in terms of truth predicates. In deference to \cite{kossak_note}, we will follow the former tradition. However, to make our results easily translatable to the language of truth predicates, we will make sure that  all satisfaction classes which we construct satisfy certain additional regularity properties.

\begin{definition} \label{def_comp}
	Let $M \models \PA$ and let $\phi \in \form_{\LPA}(M)$. Let $S \subset M^2$. We say that $S$ satisfies \df{compositional conditions} for $\phi$, $\Comp(\phi)$ if the disjunction of the following clauses holds in $(M,S)$:
	\begin{enumerate}
		\item $\exists s,t \in \Term_{\LPA} \ \Big(\phi = (s=t) \wedge \forall \alpha \in \Asn(\phi) \ \Big(S(\phi,\alpha) \equiv s^{\alpha} = t^{\alpha} \Big) \Big).$
		\item $\exists \psi \in \form_{\LPA} \ \Big(\phi = (\neg \psi) \wedge \forall \alpha \in \Asn(\phi) \ \Big(S(\phi,\alpha) \equiv \neg S(\psi,\alpha)\Big)\Big).$
		\item $\exists \psi, \eta \in \form_{\LPA}(M) \ \Big(\phi = (\psi \vee \eta) \wedge \forall \alpha \in \Asn(\phi) \Big(S(\phi,\alpha) \equiv S(\psi,\alpha) \vee S(\eta,\alpha) \Big)\Big).$
		\item $\exists \psi, \eta \in \form_{\LPA}(M) \ \Big(\phi = (\psi \wedge \eta) \wedge \forall \alpha \in \Asn(\phi) \Big(S(\phi,\alpha) \equiv S(\psi,\alpha) \wedge S(\eta,\alpha) \Big)\Big).$
		\item  $\exists \psi \in \form_{\LPA}(M) \exists v \in \Var \ \Big(\phi = (\exists v \psi) \wedge \forall \alpha \in \Asn(\phi) \Big(S(\phi,\alpha) \equiv \exists \beta \sim_v \alpha S(\psi, \beta)\Big)\Big).$  
		\item  $\exists \psi \in \form_{\LPA}(M) \exists v \in \Var \ \Big(\phi = (\forall v \psi) \wedge \forall \alpha \in \Asn(\phi) \Big(S(\phi,\alpha) \equiv \forall \beta \sim_v \alpha \ S(\psi, \beta)\Big)\Big).$  
 	\end{enumerate}  
\end{definition}

\begin{definition} \label{def_satisfaction_class}
	Let $M \models \PA$. We say that a set $S \subset M^2$ is a  \df{satisfaction class} if  there exists $D \subset  \form_{\LPA}(M)$ such that:
	\begin{itemize}
		\item If $(\phi,\alpha) \in S$, then $\phi \in \form_{\LPA}(M), \alpha \in \Asn(\phi)$.
		\item If there exists $\alpha$ such that $(\phi,\alpha) \in S$, then $(M,S) \models \Comp(\phi)$.
		\item If there exists $\alpha$ such that $(\phi,\alpha) \in S$, then either $\phi \in D$ or there exists $\psi$ such that $\phi = \neg \psi$ and $\psi \in D$. 
		\item  $(M,S)$ satisfies $\Comp(\phi)$ for every $\phi \in D$.
		\item If $\psi$ is a direct subformula of $\phi$ and $\phi \in D$, then $\psi \in D$. 
		\item If $\phi \in D$, then for every $\alpha \in \Asn(\phi)$ either $(\phi,\alpha) \in S$ or $(\neg \phi,\alpha) \in S$. 
	\end{itemize}
	 By a \df{domain of} $S$, $\dom(S)$, we mean the minimal set $D \subset \form_{\LPA}(M)$ satisfying the above conditions.\footnote{Without the minimality requirement the domain is "almost" unique, since it not determined how to treat the formulae which are negations satisfied under all assignments. If $(\neg \phi, \alpha) \in S$ for all $\alpha$ and $(\neg \neg \neg \phi, \alpha) \notin S$ for any $\alpha$, then we stipulate that $\phi$ is in the domain of $S$, but $\neg \phi$ is not.}
	We say that a satisfaction class $S$ \df{decides} a formula $\phi$ if for every $\alpha \in \Val(\phi)$ either $(\phi,\alpha) \in S$ or $(\neg \phi,\alpha) \in S$.
	 We say that $S$ is a \df{full satisfaction class} if $\dom(S) = M$, i.e., $(M,S)$ satisfies $\Comp(\phi)$ for every $\phi \in \form_{\LPA}(M)$. We call a satisfaction class \df{partial} if it is not full. 
	 A (full or partial) satisfaction class $S$ is \df{inductive} iff $(M,S)$ satisfies the full induction scheme.
\end{definition}

\begin{remark}
	Our definition of a satisfaction class agrees with the usual one for full satisfaction classes and essentially agrees with the usual one for satisfaction classes defined on the formulae in some initial segment of a model. 
	
	According to the definition proposed above, a satisfaction class is required to be completely determined about any formula to which it ascribes any assignments. If $(\phi,\alpha) \in S$, then we fully know for any $\beta \in \Asn(\phi)$ whether $(\phi,\beta)$ holds or not and the same holds for the direct subformulae of $\phi$.

	Notice that our definition is not very restrictive. Let us introduce a notion of a pre-satisfaction class: $S$ is a pre-satisfaction class if it satisfies compositional conditions on a set of formulae $D$ closed under direct subformulae. Then given such an $S$, we can define a satisfaction class $S'$ such that $S' \cap (D \times M) = S \cap (D \times M)$. Indeed, let $(\phi,\alpha) \in S'$ if either $\phi \in D$ and $(\phi,\alpha) \in S$ or $\phi = \neg \psi$, $\psi \in D$ and $(\psi,\alpha) \notin S$. One can verify that $S'$ defined in such a way is a satisfaction class in our sense, so we do not lose very much by imposing the "definiteness" conditions.
\end{remark}

\begin{remark}
	In Definition \ref{def_satisfaction_class}, we required that satisfaction classes are defined over the arithmetical language. Of course, nothing prevents us from defining them over an arbitrary (say, primitive recursive) language $\Lang$. In such a case, we have to add the atomic clauses for the new atomic formulae of $\Lang$ to the condition $\Comp(\phi)$ and replace every mention of $\form_{\LPA}$ with $\form_{\Lang}$.
\end{remark}

As we have already remarked, there are two concurrent terminological traditions which speak of essentially the same objects: satisfaction classes and truth predicates. Unfortunately, these two notions do not translate directly into each other. Indeed, suppose that $M \models \PA$ and $S$ is a full satisfaction class on $M$. Then we would like to define a truth predicate $T$ as follows:
\begin{displaymath}
T = \set{\phi[\alpha] \in \Sent_{\LPA}(M)}{(\phi,\alpha) \in S}.
\end{displaymath}
However, defined in such a way, $T$ is not necessarily a truth predicate.\footnote{We will shortly define what we actually mean by a truth predicate. For the time being, the reader should think of it as a set of sentences in the sense of a model which satisfies compositional conditions similar to that of a satisfaction class, where we  require that true quantified statements are witnessed by (certain) terms.} Indeed, take a formula $\phi(v) \in \form_{\LPA}(M)$ with only one free variable $v$ and suppose that $(\phi,\alpha) \in S$, where $\alpha$ sends $v$ to $0$. Then, by definition $T(\phi(0))$ holds. However, notice that nothing in the definition of the satisfaction class guarantees that $(\phi(0),\emptyset) \in S$. In fact, one can construct full satisfaction classes such that both $(\phi,\alpha)$ and $(\neg \phi(0), \emptyset)$ belong to  $S$. Such a class would yield $T(\neg \phi(0))$, contradicting compositionality.\footnote{Such a construction could be performed by an application of the Enayat--Visser technique which we discuss in this article.} 

A familiar definition which says that a sentence is true iff it is satisfied under any assignment works even worse. Again, we can construct models of $\PA$ with full satisfaction classes, $(M,S)$, where for some sentence $\phi \in \Sent_{\LPA}(M)$  neither $(\phi,\alpha) \in S$  nor $(\neg \phi,\alpha) \in S$ holds for all $\alpha$ (that is, the satisfaction predicate may depend nontrivially on the assignments, even though a formula in question has no free variables). A truth predicate obtained from such a model with this classical definition would not even satisfy the compositional condition for the negation.

The question of defining a satisfaction class out of a truth class is slightly subtler and depends on what exact compositional conditions for the quantifiers are accepted. One of the common choices requires that quantified sentences are true if they are witnessed by some term. Thus the compositional axiom for the existential quantifier looks as follows:
\begin{displaymath}
T \exists v \phi \equiv \exists t \in \ClTerm_{\LPA} \ T\phi(t).
\end{displaymath}
If we accept this kind of axioms and set, for instance,
\begin{displaymath}
S = \set{(\phi,\alpha) \in M^2}{\phi[\alpha] \in T}.
\end{displaymath}
then in general we cannot show that quantifier axioms for the satisfaction class will hold, since it could happen for a nonstandard $\phi \in \form_{\LPA}(M)$ that $\exists v \phi \in T, \phi(0+0) \in T$ and for all $x\in M$, $\phi(\num{x}) \notin T$. Then the quantifier axiom for $S$ would break down as well.  (Recall that in $\phi[\alpha]$ we substitute for any free variable $v$ the \emph{numeral} whose value is $\alpha(v)$.)  Whether a satisfaction class can define a truth predicate with some other, possibly much more complex definition (or vice versa), seems to be a rather subtle question, but we will not pursue it any further in this article.

In order to make the statements about truth predicates and satisfaction classes directly translatable between the two settings, we will impose some additional technical conditions on the constructed classes. This will not be important for any results in this article (i.e., they still hold if we do not make any of these regularity assumptions). 

\begin{definition} \label{defi_extensional_equivalence}
	Suppose that $\phi, \psi$ are arithmetical formulae. Let $\alpha \in \Asn(\phi), \beta \in \Asn(\psi)$. We say that the pairs $(\phi,\alpha), (\psi,\beta)$ are \df{extensionally equivalent} if there exist two sequences of closed terms $\bar{s}, \bar{t}$ such that for each $i$, $\val{t_i} = \val{s_i}$, and a formula $\eta$ such that $\phi[\alpha] = \eta(\bar{s}), \psi[\beta] = \eta(\bar{t})$. If $(\phi,\alpha), (\psi,\beta)$ are extensionally equivalent, we denote it with $(\phi,\alpha) \simeq (\psi,\beta)$. 
\end{definition}

\begin{definition} \label{defi_regular_satisfaction}
Let $M \models \PA$ and let $S \subset M^2$ be a (full or partial) satisfaction class. We say that $M$ is \df{regular} if for any extensionally equivalent pairs $(\phi,\alpha) \simeq (\psi,\beta)$, $(\phi,\alpha) \in S$ iff $(\psi,\beta) \in S$. 
\end{definition}

Regular satisfaction classes are in one-to-one correspondence with compositional truth predicates satisfying regularity assumptions, so our results extend to that setting. Let us state this in a more precise manner.

\begin{definition} \label{def_ctminus}
	By $\CT^-$ (compositional truth) we mean an axiomatic theory in the arithmetical language with a fresh predicate $T$ extending $\PA$ with the following axioms:
	\begin{enumerate}
		\item $\forall x \ \Big(T(x) \rightarrow x \in \Sent_{\LPA} \Big).$
		\item $\forall s,t \in \ClTerm_{\LPA} \ \Big(T(s=t) \equiv \val{s} = \val{t}\Big).$
		\item $\forall \phi \in \Sent_{\LPA} \ \Big(T \neg \phi \equiv \neg T \phi \Big).$
		\item $\forall \phi, \psi \in \Sent_{\LPA} \ \Big(T (\phi \vee \psi) \equiv T\phi \vee T \psi )\Big).$
		\item $\forall \phi, \psi \in \Sent_{\LPA} \ \Big(T (\phi \wedge \psi) \equiv T\phi \wedge T \psi )\Big).$
		\item $\forall \phi \in \Sent_{\LPA} \forall \psi \in \form_{\LPA} \forall v \in \Var \ \Big( \phi = (\exists v \psi) \rightarrow \Big(T \phi \equiv \exists x \ T \phi(\num{x})\Big) \Big).$
		\item $\forall \phi \in \Sent_{\LPA} \forall \psi \in \form_{\LPA} \forall v \in \Var \ \Big( \phi = (\forall v \psi) \rightarrow \Big(T \phi \equiv \forall x \ T \phi(\num{x})\Big) \Big).$
		\item $\forall \phi, \psi \in \form_{\LPA} \forall \bar{s}, \bar{t} \in \ClTermSeq \Big(\bar{\val{s}} = \bar{\val{t}} \rightarrow T \phi(\bar{s}) \equiv T \phi (\bar{t})\Big).$
	\end{enumerate}
\end{definition}

Essentially, in our definitions we stipulate that both satisfaction classes and truth predicates are not sensitive to specific terms, but only see their values. The following claim can be now easily verified. 
\begin{proposition} \label{prop_satisfaction_vs_truth}
	Let $M \models \PA$ and let $S \subset M^2$ be a regular satisfaction class. Let $T := \set{\phi \in \Sent_{\LPA}(M)}{(\phi[\alpha],\emptyset) \in S}$. Then $(M,T) \models \CT^-$. 	Conversely, let $(M,T) \models \CT^-$. Let $S = \set{(\phi,\alpha) \in M^2}{\phi[\alpha] \in T}.$ Then $S$ is a regular satisfaction class.
\end{proposition}

We can define a number of other natural regularity properties of satisfaction or truth classes. Let us introduce one of them which will play a prominent role in our paper. 

\begin{definition} \label{defi_quantifier correctness}
	Let $M \models \PA$, let $S$ be a satisfaction class, and let $\phi \in \form_{\LPA}(M)$. We say that the \df{quantifier correctness} property holds for $\phi$ if for all $c \in M$ and all $\bar{v} \in \VarSeq(M)$ such that $\exists \bar{v} \phi$ is in the domain of $S$, 
	\begin{displaymath}
	S(\exists \bar{v} \phi,\alpha) \equiv \exists \beta \sim_{\bar{v}} \alpha \ S(\phi,\beta);
	\end{displaymath} 
	and for all $\bar{v} \in \VarSeq(M)$ such that $\forall \bar{v} \phi$ is in the domain of $S$, 
	\begin{displaymath}
	S(\forall \bar{v} \phi,\alpha) \equiv \forall \beta \sim_{\bar{v}} \alpha \  S(\phi,\beta).
	\end{displaymath}
	
	We say that the quantifier correctness \df{fails} for $\phi$ if one of the above equivalences does not hold for a pair of formulae $\phi, \exists \bar {v} \phi$ or $\phi, \forall \bar{v} \phi$ which are both in $\dom(S)$.
\end{definition}
In other words, quantifier correctness means that $\phi$ satisfies compositional clauses with respect to blocks of quantifiers. If $S$ is a full satisfaction class, then the quantifier correctness fails for $\phi$ iff it does not hold for $\phi$. In the case of partial satisfaction classes the distinction is nontrivial. For instance, suppose that $M \models \PA$ is a nonstandard model and $S \subset M^2$ contains only standard arithmetical formulae (on which all satisfaction classes agree). Then the quantifier correctness will not hold in general, since $S$ does not contain e.g. $\exists v_1 \ldots \exists v_c 0=0$ for nonstandardly  long blocs of quantifiers. On the other hand, this is not yet a failure of the quantifier correctness for which we require that for some $\alpha$, $S(\neg \exists v_1 \ldots \exists v_c 0=0, \alpha)$ holds.

\section{Enayat--Visser Lemma}

In the article \cite{enayatvisser2} (see also \cite{enayat_visser_long}), a method for constructing satisfaction classes was introduced which proved to be extremely handy, as it allows us to guarantee that the constructed classes enjoy a number of additional properties. 

However, these applications typically require slight modifications of the original proof in a way which prevents us from simply quoting the result. One attempt to present Enayat--Visser construction in a form of a more general statement has been made by James Schmerl in \cite{schmerl_kernels}. Unfortunately, his formulation, while very elegant, seems (at least to us) rather unsuitable for the typical applications of the discussed technique. In this section, we try to present the core of Enayat--Visser argument in a way that allows us to use it as a (simple) black box.

\begin{definition} \label{def_locally_admissible_properties}
	By an \df{Enayat--Visser class} (or EV class) we mean a class $\mathscr{C}$ of pairs $(M,S)$, where $M$ is a model of some theory $T$ in a language extending $\Lang$ and $S \subset M^2$ is an $\Lang$-satisfaction class such that the following conditions are satisfied:
	\begin{enumerate}
		\item For every $(M,S) \in \mathscr{C}$, there exists $(M',S') \in \mathscr{C}$ such that $M' \succeq M$, $S' \supseteq S$, and the the domain of $S'$ contains $\form_{\Lang}(M)$. (Unboundedness.)
		\item If $(M_i,S_i) \in \mathscr{C}$ for $i < \omega$ and for each $i$, $M_i \preceq M_{i+1}, S_i \subset S_{i+1}$, then $(M',S') \in \mathscr{C}$, where $M' = \bigcup_{i < \omega} M_i, S' = \bigcup_{i < \omega} S_i$. (Closure.)
	\end{enumerate}
\end{definition}

In the paper, we will only focus on the case where $\Lang = \LPA$ is the arithmetical language. Note that in the definition we do not assume that $\mathscr{C}$ is nonempty. However, it will be implied by the assumptions of the following theorem, which encapsulates the core of the Enayat--Visser argument.

\begin{theorem}[Enayat--Visser Lemma] \label{th_enayat_visser_lemma}
	Let $\mathscr{C}$ be an Enayat--Visser class, let $M \models T$, and suppose that there exists $S \subset M^2$ such that $(M,S) \in \mathscr{C}$. Then there exists a model $M' \succeq M$ and a full $\mathscr{L}$-satisfaction class $S' \supset S $ such that $(M',S') \in \mathscr{C}$. 
\end{theorem} 
\begin{proof}
	Fix a model $M \models T$, and a satisfaction class $S$ on $M$ such that $(M,S) \in \mathscr{C}$. Using the second condition of the definition of Enayat--Visser classes, we inductively build a chain of models $M=M_0 \preceq M_1 \preceq M_2 \preceq \ldots$ of length $\omega$ and satisfaction classes $S_i \subset M_i$ such that for each $i<j$, $S_i \subseteq S_j$, the domain of $S_{i+1}$ contains $\form_{\Lang}(S_i)$, and $(M_i,S_i) \in \mathscr{C}$.
	
	Let $M' = \bigcup_{i<\omega} M_i$ and let $S' = \bigcup_{i< \omega} S_i$. By the third condition in the definition of Enayat--Visser class, $(M',S') \in \mathscr{C}$, so it is enough to verify that $S'$ is a  full satisfaction class. Since every $\phi \in \form_{\Lang}(M')$ is in $\form_{\Lang}(M_j)$ for some $j$, this implies that $\phi$ is in the domain of $S_{j+1}$ and, consequently, in the domain of $S'$. 
\end{proof}

The essence of Enayat--Visser technique is that the unboundedness condition in the definition of EV is much easier to check than the existence of a full satisfaction class $S'$, since the former can be expressed as a scheme which in turn allows us to use a compactness argument. We will use this method in Section \ref{sec_proof_of_induction_step}. Let us add that it would be nice to incorporate this "local" nature of the argument to the statement of Theorem \ref{th_enayat_visser_lemma}. Unfortunately, it seems that general enough statements which cover actual applications of the technique would be somewhat awkward to express.

\section{Full satisfaction classes and definability} 

In this section, we will prove that in every countable, recursively saturated model $M \models \PA$, there exists a full satisfaction class $S$ such that $(M,S)$ is minimal. We will actually demonstrate a more general result.

\begin{theorem} \label{tw_every_set_definable_from_satisfaction}
	Let $M \models \PA$ be a countable recursively saturated model and let $A \subset M$ be an arbitrary set. 	There exists a full satisfaction class $S \subset M^2$ such that $A$ is definable in $(M,S)$ without parametres. 
\end{theorem} 

\begin{proof}
	Let $M \models \PA$ be a countable recursively saturated model. Let for $a \in M$, 
	\begin{displaymath}
	\eta_{a} = \forall x \underbrace{\Bigl( x=x \wedge x=x \wedge \ldots \wedge x=x \Bigr)}_{a \textnormal{ times}}
	\end{displaymath}
	 and let for $\bar{v} \in \VarSeq(M)$, $\eta_{a,\bar{v}}$ be the following formula:
	\begin{displaymath}
	\exists \bar{v} \ \eta_a.
	\end{displaymath} 
We will construct a satisfaction class such that:
	\begin{itemize}
		\item If quantifier correctness fails for $\phi$, then $\phi = \eta_{a,\bar{v}}$ for some $a \in A$ and some $\bar{v} \in  \VarSeq(M)$ (where we also allow the trivial block of quantifiers). 
		\item If $a \in A$, then the quantifier correctness fails for $\eta_a$.
	\end{itemize} 
Let us fix an enumeration $(a_n)_{n \in \omega}$ of the elements of $M$ and a cofinal increasing sequence of elements of $M$, $(b_n)_{n \in \omega}$.
  We inductively construct regular satisfaction classes $S_0 \subseteq S_1 \subseteq \ldots \subseteq M$ such that:
	\begin{itemize}
		\item The domain of $S_n$ contains exactly the formulae of syntactic depth at most $b$ for some $b \geq b_n$.\footnote{Recall that the notion of syntactic depth was introduced in Definition \ref{def_formalised_syntax}.}
		\item If $a_n \in A$, then the quantifier correctness fails for the formula $\eta_{a_n}$.
		\item If $\phi \in \form_{\LPA}(M)$ is not of the form $\eta_{a_i,\bar{v}}$ for some $i \leq n, \bar{v} \in \VarSeq(M)$, then the quantifier correctness holds for $\phi$. 
		\item The model $(M,S_n)$ is recursively saturated.
	\end{itemize}
The inductive steps of the construction are handled via the following lemma:
\begin{lemma}[Flexibility of quantifier correctness] \label{lem_krok_indukcyjny_poprawnosc_kwantyfikatorowa}
	Let $M \models \PA$. Let $a,b \in S$ and let $S \subset M$ be a regular  satisfaction class whose domain consists exactly of formulae of syntactic depth at most $b$. Then there exists an elementary extension $(M,S) \preceq (M',S'_0)$ and a full satisfaction class $S' \supset S'_0$ such that the quantifier correctness fails for the  formula $\eta_{a}$ and if it fails for $\phi$ according to $S'$, then  it fails according to $S_0'$ or $\phi = \eta_{a,\bar{v}}$ for some $\bar{v} \in \VarSeq(M')$.
\end{lemma}
The above Lemma  can be proved by the Enayat--Visser methods, but the proof itself turns out to be somewhat complicated when we are trying to provide all the relevant details which would be rather disruptive for the whole argument. Therefore we postpone it until Section \ref{sec_proof_of_induction_step}. Let us now finish the proof under the assumption that the lemma holds. Suppose that we have constructed the set $S_n$.  Applying Lemma \ref{lem_krok_indukcyjny_poprawnosc_kwantyfikatorowa} to the model $(M,S_n)$ and using its resplendence, we find a full satisfaction class $S \supseteq S_n$ such that $(M,S)$ is recursively saturated as a structure in the extended language, the quantifier correctness fails for formulae of the form $\eta_{a_i}$ for $i \leq n+1$ such that $a_i \in A$ and holds for all formulae not of the form $\eta_{a_i,\bar{v}}$. Finally, we let $S_{n+1}$ be $S$ restricted so that $\dom (S_{n+1})$ contains exactly the formulae of syntactic depth  at most $b  \geq b_{n+1}$, where $b$ is large enough that if the quantifier correctness fails for $\eta_{a_i}$, $i \leq n+1$, then this failure is already witnessed by the formulae of depth less than $b$. Since this restriction is definable in $S$, the obtained model $(M,S_{n+1})$ is also recursively saturated and clearly satisfies all the other requirements of the inductive construction.

We set $S = \bigcup_{n \in \omega} S_n$. Since all $S_n$ are compatible satisfaction classes, and the union of their domains is the whole model, $S$ is a full satisfaction class as well.  Finally, we claim that $A$ is definable in $(M,S)$. Indeed, it is enough to notice that $A$ is precisely the set of $x$ such that the quantifier correctness fails for the formula $\eta_x$.  	
\end{proof}

The following corollary answers Question 9.5 from \cite{enayat_visser_long} in the positive.

\begin{corollary} \label{cor_satisfaction_classes_where_every_element_definable}
	Let $M$ be a countable recursively saturated model of $\PA$. Then there exists a full satisfaction class $S \subset M$ such that every element of $M$ is definable in $(M,S)$ without parametres.
\end{corollary}

\begin{proof}
	Fix any bijection $f: \omega \to M$. This bijection can be coded by a subset of $M^2$. Fix a satisfaction class $S \subset M$ such that $f$ is definable in $M$. Then every element of $M$ is definable in $(M,S)$, as it is equal to $f(n)$ for some $n \in \omega$.
\end{proof}

\begin{corollary} \label{cor_minimal_rigid_satisfaction_classes}
	For any countable recursively saturated model $M \models \PA$, there exists a full satisfaction class $S \subset M$ such that $(M,S)$ has no proper elementary submodels and no nontrivial automorphisms.
\end{corollary}
\begin{proof}
	This follows immediately from the previous corollary.
\end{proof}

\begin{corollary} \label{cor_continuum_many_nonequivalent_satisfaction}
	In every countable recursively saturated model $M \models \PA$, there exist continuum many pairwise non-elementary-equivalent full satisfaction classes.
\end{corollary}
\begin{proof}
	Let $M \models \PA$ be countable and recursively saturated.	By inspection of the proof of Theorem \ref{tw_every_set_definable_from_satisfaction}, we see that there exists a single formula such that for every subset $A \subseteq M$, there exists a full satisfaction class $S_A \subset M^2$ such that $A$ is definable in $(M,S)$ with that formula. Consequently, if $(M,A)$ and $(M,A')$ are not elementarily equivalent, then neither are $(M,S_A), (M,S_{A'})$.
\end{proof}

It is natural to ask whether analogues of Theorem \ref{tw_every_set_definable_from_satisfaction} hold also in higher cardinalities. Unfortunately, we cannot immediately see whether our results generalise to uncountable models $M \models \PA$, even if we assume that the said models are, for instance, chronically resplendent. We can prove a weaker result in the similar vein.

\begin{corollary}
	If $M$ is any model of $\PA$ (of an arbitrary cardinality), then for any $A \subseteq M$, there exists an elementary extension $(M',A') \succeq (M,A)$ and a satisfaction class $S' \subset {M'}^2$ such that $A'$ is definable in $(M',S')$.
\end{corollary} 
\begin{proof}
	Take a countable recursively saturated model $(M_0,A_0) \equiv (M,A)$. By applying Theorem \ref{tw_every_set_definable_from_satisfaction} to $(M_0,A_0)$, we see that the elementary diagram of $(M,A)$ is consistent with the statement "$A$ is definable from a satisfaction class $S$ via a formula $\phi$," where $\phi$ is implicitly given in the proof of the theorem. The corollary follows.
\end{proof}

 In \cite{smith}, Theorem 3.3, another result similar to Theorem \ref{tw_every_set_definable_from_satisfaction} was obtained. 
\begin{theorem}[Smith] \label{th_smith_definability}
	Let $M \models \PA$ be countable and recursively saturated. Let $A \subset M$ be such that the expanded structure $(M,A)$ is recursively saturated. Then there exists a full satisfaction class $S$ over $M$ and $\phi(v) \in \form_{\LPA}(M)$ such that
	\begin{displaymath}
	A = \set{x \in M}{(\phi(v),\tuple{v,x}) \in S}.
	\end{displaymath}
	Moreover, there exists a uniformly definable family of formulae $(\varepsilon_a(v))_{a \in M}$ such that for each $a \notin \omega$, there exists a satisfaction class $S$ for which
	\begin{displaymath}
	A = \set{x \in M}{(\varepsilon_a(v),\tuple{v,x}) \in S}
	\end{displaymath}
\end{theorem}
By writing $\phi(v)$, we mean that $v$ is a free variable of $\phi$. Notice that $\tuple{v,x}$ is an assignment which ascribes the value $x$ to the only variable of $\phi$. The original proof relied on the $M$-provability methods. A closely related result obtained via the Enayat--Visser  construction (and thus possibly easier to digest) can be found in \cite{WcisloKossak}.

Let us make one final remark. One could be dissatisfied with the fact that the satisfaction class we have constructed is rather pathological and one could wonder whether the result would still hold if we restricted our attention to full satisfaction classes which satisfy the quantifier correctness property.

It seems that while we used some specific pathology, the way we did it is rather generic. For instance if we require that the constructed class satisfy the quantifier correctness, we could turn our attention to disjunctions of the form $\eta_a \vee \eta_a \vee \ldots \vee \eta_a$ and ask for which $a$ all such disjunctions are true.\footnote{A similar flexibility property was used independently by Mateusz Łełyk in his unpublished work on Tarski Boundary.} If we additionally require that all sentences true in the sense of some partial arithmetical $\Sigma_n$ predicate  are rendered true by our satisfaction class, we can consider the set of sentences such that for every $c$, the sentence $\forall v \xi_{a,c}$ is true, where the formulae $\xi_{a,c}$ are defined inductively via the conditions:
	\begin{eqnarray*}
		\xi_{a,0} & = & v=v \wedge \eta_{a} \\
		\xi_{a,b+1} & = & \forall x_b \exists y_b \ (x_b= v \wedge y_b = v \wedge \xi_{a,b}).
	\end{eqnarray*} 
With just a slightly more complicated proof than the one presented above, one could show that the formulae $\xi_{a,c}$ are also flexible in the sense that the set of those $a$ such that $\xi_{a,c}$ is true for all $c$ can be an arbitrary set.

 On the other hand, we always have to expect that full satisfaction classes display some pathological behaviour, since already such benign principles as "a finite disjunction is satisfied iff one of the disjuncts is," "truth is closed under propositional proofs" or "every first-order tautology is true" already allow us to draw nontrivial conclusions about the arithmetical theory of the underlying model.\footnote{See \cite{EnayatPakhomov} for the result on disjunctions and \cite{cies_ksiazka} for the other proofs and an extensive discussion of the Tarski Boundary programme which investigates this phenomenon.}

  We conjecture that if a satisfaction class displays some pathological behaviour, then those pathologies are flexible enough to define any subset of the model (which will be slightly easier to phrase in terms of truth predicates, as they are explicitly presented as axiomatic theories).

\begin{question}
	Suppose that $P \supseteq \CT^-$ is a finite $\Lang_{\CT^-}$-theory conservative over $\PA$. Let $M \models \PA$ be a countable recursively saturated model. Is every set $A \subset M$ definable in some structure of the form $(M,T) \models P$?
\end{question}

\section{Satisfaction classes and automorphisms}	

In the previous section, we have shown that in every countable recursively saturated model $M \models \PA$ there exists a rigid satisfaction class, i.e. a satisfaction class $S \subset M$ such that $(M,S)$ admits no non-trivial automorphisms. A natural question arises whether the opposite result holds, i.e., whether given a countable recursively saturated model $M$, we can find a satisfaction class $S$ on $M$ such that $\Aut(M,S) = \Aut(M)$.

Clearly, if $f \in \Aut(M)$, then for any $a \in M$, the elements $a$ and $f(a)$ have the same arithmetical types. On the other hand, if $a,b \in M$ have the same arithmetical types and $M$ is countable and recursively saturated, then by Theorem \ref{th_recursive_saturated_automorphism}, there exists an automorphism $f \in \Aut(M)$ such that $f(a) = b$. Therefore, it is enough to check whether there exists a satisfaction class $S$ on $M$ such that for any two pairs $(\phi,\alpha), (\psi,\beta)$ with the same arithmetical types, $(\phi,\alpha) \in S$ iff $(\phi',\alpha') \in S$.\footnote{We are grateful for this observation to Roman Kossak who also suggested to study automorphisms and definability properties of satisfaction classes.}

One could guess that such a satisfaction class can be easily constructed by an application of the Enayat--Visser techniques. It is rather surprising that this intuition is wrong. 

\begin{theorem} \label{th_satisfaction_never_type_congruent}
	Let $M \models \PA$ and let $S \subset M^2$ be a full regular satisfaction class. Then there exist two pairs $(\phi,\alpha), (\phi',\alpha') \in M^2$ such that their arithmetical types are equal, but exactly one of them is in $S$. 
\end{theorem}
The proof will use the following fact:

\begin{proposition} \label{lem_inductive_satisfaction_never_type_congruent}
	Let $M \models \PA$ and let $S \subset M^2$ be a regular inductive partial satisfaction class containing all formulae of standard syntactic depth. Then there exist two sentences $\phi,\psi \in \Sent_{\LPA}(M)$ such that $\tp(\phi) = \tp(\psi)$, but $(\phi,\emptyset) \in S, (\psi,\emptyset) \notin S$.
	
	Moreover, for any element $a \in M$, there exists a pair of sentences $\phi,\psi \in \form_{\LPA}^{\leq 1}(M)$ such that $\tp (\phi,a) = \tp (\psi,a)$, but $(\phi,\emptyset) \in S, (\psi,\emptyset) \notin S.$\footnote{This Proposition essentially appears  in the second proof of Theorem 1 in \cite{HamkinsYang}, attributed to Woodin, where it is stated and proved in the context of recursively saturated satisfaction classes, rather than inductive ones. It is also formulated as a remark in a MathOverflow question, \url{https://mathoverflow.net/questions/186823/is-there-a-nonstandard-model-of-arithmetic-having-precisely-one-inductive-truth} although the argument seems to contain a gap. However, it is exactly that remark which inspired the results in this section. Previously, we expected that a satisfaction class can in fact have the same automorphism group as the original model.}
\end{proposition}

In the proof below, by an \df{$n$-type} of an element $a$, we mean the set of arithmetical formulae in a single variable of syntactic depth $n$ satisfied by $a$. We denote the $n$-type of $a$ by $\tp_n(a)$. It is easy to see that over relational languages the set of $n$-types is finite for each $n$. 	In this paper, we actually assume that $\LPA$ contains function symbols. However, we can formulate a relational language, where each $n$-ary arithmetical function symbol is replaced with an $n+1$-ary relation symbol and translate every sentence in $\PA$ to an equivalent sentence in that language, so we can indeed assume that for each $n$, the set of arithmetical $n$-types over $\PA$ is finite.

\begin{proof}	
	To prove the first part of the claim, notice that for every $n \in \omega$, it is expressible in $\PA$ that two elements $a,b$ have the same arithmetical $n$-type.  Therefore, for every $n \in \omega$, there exist two sentences $\phi, \psi \in \Sent_{\LPA}(M)$, such that $\tp_n(\phi) = \tp_n(\psi)$, but only one of them is in $S$, since otherwise $S$ would be arithmetically definable over $M$ (a sentence would be true according to $S$ if its $n$-type is one of the finitely many $n$-types which guarantee that the sentence is true).
	
	By overspill, there exists a pair of sentences $\phi,\psi \in \Sent_{\LPA}(M)$ and a nonstandard element $c$ such that 
	\begin{displaymath}
	(M,S) \models \forall n \leq c \ S(\tp_n(\num{\phi}) = \tp_n (\num{\psi}), \emptyset) \wedge S(\phi,\emptyset) \wedge \neg S(\psi,\emptyset).
	\end{displaymath}
	These sentences $\phi,\psi$ satisfy our requirements. 
	
	To prove the "moreover" part, notice that we can consider sentences of the form $\phi^{a}$ defined as  $\num{a} = \num{a} \wedge \phi$. Again, for any $n \in \omega$, there exists a pair of sentences $\phi^a, \psi^a \in \Sent_{\LPA}(M)$ such that they have the same $n$-types and exactly one is in $S$, since otherwise $S$ (and consequently, the elementary diagram of $M$) would be definable in $M$ over the parameter $a$. Just like in the previous case, we can find two sentences $\phi^a, \psi^a$ which have the same arithmetical types, but only one of them is true according to $S$. Finally, observe that $a$ is definable from both $\phi^a$ and $\psi^a$, so if $\tp(\phi^a) = \tp(\psi^a)$, it follows that $\tp(\phi,a) = \tp(\psi,a)$.
\end{proof}

Now we are ready to prove Theorem \ref{th_satisfaction_never_type_congruent}. Our argument will make use of the following fact. Its proof may be found in \cite{WcisloKossak}, Theorem 3.

\begin{theorem} \label{th_inductive_satisfaction_definable_in_full_satisfaction}
	Suppose that $M \models \PA$ and $S \subseteq M^2$ is a regular full satisfaction class. Then there exists $S' \subset M^2$ such that $S'$ is a regular inductive satisfaction class whose domain (properly) contains all formulae of standard syntactic depth.

	Moreover, there exists $\gamma(v) \in \form_{\LPA}(M)$ such that $S'$ is definable in $(M,S)$ with the parameter $\gamma$ via:
	\begin{displaymath}
	S' = \set{(\phi,\alpha) \in M^2}{\phi \in \form_{\LPA}(M), \alpha \in \Val(\phi) \wedge (\gamma(\num{\phi[\alpha]}), \emptyset) \in S}.
	\end{displaymath} 
\end{theorem}

\begin{proof}[Proof of Theorem \ref{th_satisfaction_never_type_congruent}]
	Let $M \models \PA$ and let $S$ be a full satisfaction class on $M$. Let $S'$ be the inductive satisfaction class on $M$ definable from $(M,S)$ with the parameter $\gamma$, as in Theorem \ref{th_inductive_satisfaction_definable_in_full_satisfaction}. 
	
	Let $\phi, \psi \in \Sent_{\LPA}(M)$ be two elements such that $(\phi,\emptyset) \in S', (\psi,\emptyset) \notin S'$ and $\tp(\phi, \gamma) = \tp (\psi,\gamma)$, as given by Lemma \ref{lem_inductive_satisfaction_never_type_congruent}.  By definition of $S'$, it follows that 
	\begin{displaymath}
	(\gamma(\num{\phi}),\emptyset) \in S, (\gamma(\num{\psi}),\emptyset) \notin S.
	\end{displaymath}
	Since $\tp(\phi,\gamma) = \tp(\psi,\gamma)$, it follows that $\tp(\gamma(\num{\phi})) = \tp (\gamma(\num{\psi}))$, because substitution is a definable operation. Thus $(\gamma(\num{\phi}), \emptyset), (\gamma(\num{\psi}),\emptyset)$ are the desired pairs. 
\end{proof}

\begin{remark}
	Actually, the regularity requirement is not needed in Theorem \ref{th_inductive_satisfaction_definable_in_full_satisfaction} and consequently in Theorem \ref{th_satisfaction_never_type_congruent}. However, the cited result is formulated in \cite{WcisloKossak} in the context of $\CT^-$, where we automatically assume regularity, so we decided to leave it as is. 
\end{remark}

\section{Flexibility of quantifier correctness} \label{sec_proof_of_induction_step}

In order to complete the proof of Theorem \ref{tw_every_set_definable_from_satisfaction}, we need to show that Lemma \ref{lem_krok_indukcyjny_poprawnosc_kwantyfikatorowa} holds. We restate it for convenience of the reader.
\begin{lemma} \label{lem_krok_indukcyjny_poprawnosc_kwantyfikatorowa_rep}
		Let $M \models \PA$. Let $a,b \in S$ and let $S \subset M$ be a regular  satisfaction class whose domain consists exactly of formulae of syntactic depth at most $b$. Then there exists an elementary extension $(M,S) \preceq (M',S'_0)$ and a full satisfaction class $S' \supset S'_0$ such that the quantifier correctness fails for the  formula $\eta_{a}$ and if it fails for $\phi$ according to $S'$, then  it fails according to $S_0'$ or $\phi = \eta_{a,\bar{v}}$ for some $\bar{v} \in \VarSeq(M')$.
\end{lemma}
We will show that a suitable class of models is an Enayat--Visser class. Before we proceed to the proof, let us introduce one more technical notion.

\begin{definition} \label{def_templetjki}
	Let $\phi$ be an arithmetical formula. By a \df{template} of $\phi$ we mean the least formula $\widehat{\phi}$ (with respect to the G\"odel code) such that:
	\begin{itemize}
		\item  $\phi$ can be obtained by substituting arithmetical terms in $\widehat{\phi}$.
		\item Every free variable occurs in $\widehat{\phi}$ exactly once. 
		\item No variable occurs in $\widehat{\phi}$ both free and bound.
		\item No complex term containing only free variables occurs in $\widehat{\phi}$. 
		\item No closed term occurs in $\widehat{\phi}.$ 
	\end{itemize}
	We say that formulae $\phi,\psi$ are \df{syntactically similar} if $\widehat{\phi} = \widehat{\psi}$. We denote this relation with $\phi \sim \psi.$
\end{definition}

Notice that if $\phi(\bar{s}) = \eta(\bar{t})$ and $\psi(\bar{q}) = \eta(\bar{r})$ for some sequences of terms $\bar{q},\bar{r},\bar{s},\bar{t}$, then $\phi \sim \psi$. In particular, if $(\phi,\alpha) \simeq (\psi,\beta)$ are extensionally equivalent, then $\phi$ and $\psi$ are syntactically similar, but the reverse implication does not hold. 

\begin{example}
	Let $\phi = \exists x \forall y \Big(x + (x \times 0) = (y \times v) + (v \times (w+0))\Big).$ Then
	\begin{displaymath}
	\widehat{\phi} = \exists x \forall y \Big(x + (x \times v_0) = (y \times v_1) +  v_2\Big),
	\end{displaymath}
	where $v_0,v_1,v_2$ are chosen so as to minimise the G\"odel code of $\widehat{\phi}$. Let $\alpha$ be an assignment which sends $v$ to $17$ and $w$ to $2$. Let $\beta$ be an assignment which sends $v_0$ to $0$, $v_1$ to $17$ and $v_2$ to $34$. Then $(\phi,\alpha) \simeq (\widehat{\phi},\beta)$.
\end{example}

Now, we can introduce the Enayat--Visser class relevant for Lemma \ref{lem_krok_indukcyjny_poprawnosc_kwantyfikatorowa}.

\begin{lemma} \label{lem_quantifier_correcntess_EV}
	Let $M \models \PA$, $a,b \in M$, and let $T \subset M$ be a regular satisfaction class whose domain contains exactly the formulae of syntactic depth $\leq b$. Then the class $\mathscr{C}$ is an Enayat--Visser class, where $\mathscr{C}$ consists of the pairs $((N,P),S)$ such that: 
	\begin{itemize}
		\item $(N,P) \succeq (M,T)$.
		\item $S \supset P$ is a regular satisfaction class (over $\LPA$). 
		\item The quantifier correctness fails for $\eta_a$.
		\item If the quantifier correctness fails for a formula $\phi$ according to $S$, then it fails according to $P$  or there exists $\bar{v} \in \VarSeq(N)$ such that $\phi = \eta_{a,\bar{v}}$.
		\item For every $\phi \in \dom(S)$, if $\phi = Q \bar{v} \psi$ for some $\bar{v} \in \VarSeq(N)$ and $\psi$ does not begin with $Q$, then $\psi \in \dom(S).$ 
	\end{itemize}
\end{lemma}

 The last condition is technical. It clearly holds for satisfaction classes which are defined exactly on formulae from an elementary submodel, which is exactly what we will get in our construction. Notice that if the lemma holds, then applying Theorem \ref{th_enayat_visser_lemma} to the class $\mathscr{C}$ defined above, we conclude that $\mathscr{C}$ contains a model with a full satisfaction class. We can then readily verify that such a model satisfies the claim of Lemma \ref{lem_krok_indukcyjny_poprawnosc_kwantyfikatorowa_rep}. So it is enough to prove that $\mathscr{C}$ is indeed an Enayat--Visser class.

\begin{proof}
	Fix a model $M \models \PA$, elements $a,b \in M$, and a satisfaction class $T$ on $M$ whose domain consists exactly of formulae of depth $\leq b$. We will show that the class $\mathscr{C}$ defined with respect to this model is an EV-class containing some pair $((M',T'),S)$. It is easy to check that $\mathscr{C}$ is closed under unions of chains. It is enough to observe that if a failure of quantifier correctness is witnessed by a pair of formulae in a satisfaction class $S$, then it is also witnessed in any $S'$ containing $S$. On the other hand, if quantifier correctness fails in a union of a chain of  satisfaction classes, then this failure has to be witnessed in one of the models.

	We will check the unboundedness condition and explain how the argument can be modified to show the existence of an elementary extension  $(M',T') \succeq (M,T)$ and a satisfaction class $S$ in that extension such that the pair $((M',T'),S)$ belongs to $\mathscr{C}$. In what follows, we will identify pairs $((X,Y),Z)$ with triples $(X,Y,Z)$. 
	
	Fix a tuple $(N,P,S) \in \mathscr{C}$ or $(N,P,S) = (M,T,T)$ (to find any $(M',T',S) \in \mathscr{C}$ witnessing its nonemptiness). We will find $(N',P',S') \in \mathscr{C}$ such that $(N,P) \preceq (N',P')$, $S' \supseteq S$ is a satisfaction class, and the domain of $S'$ contains $\form_{\LPA}(N)$. We will  say that a formula $\phi \in \form_{\LPA}(N)$ is \df{unproblematic}  if $\phi$ is not in the domain of $P$ and $\phi$ is not of the form $\eta_{a,\bar{v}}$ for any $\bar{v} \in \VarSeq(N)$.  The structure $(N',P',S')$ will be obtained as a model of the theory $\Theta$ with the following axioms:
	\begin{itemize}
		\item The elementary diagram of $(N,P)$, where the predicate $P$ is replaced with $P'$.
		\item $\Comp(\phi), \phi \in \form_{\LPA}(M)$ (The compositionality scheme for $S'$.)
		\item $\forall \phi, \psi \in \form_{\LPA} \forall \alpha \in \Val(\phi), \beta \in \Val(\psi) \Big((\phi,\alpha) \simeq (\psi,\beta) \rightarrow S'(\phi,\alpha) \equiv S'(\psi,\beta)\Big).$ (Regularity.)
		\item $S'(\phi,\alpha)$ for all $(\phi, \alpha) \in S.$ (The compatibility scheme.)
		\item $\forall x, y \ P'(x,y) \rightarrow S'(x,y).$
		\item $\forall \alpha \in \Asn(Q\bar{v}\phi) \ \Big(S'(Q\bar{v} \phi, \alpha) \equiv Q \beta \sim_{\bar{v}} \alpha \ S'(\phi,\beta) \Big)$, where $Q \in \{\forall,\exists\}, \bar{v} \in \VarSeq(M), \phi \in \form_{\LPA}(M)$ and $\phi$ is unproblematic. (The quantifier correctness scheme.) 
	\end{itemize} 

In the initial step, when $(N,P,S) = (M,T,T)$, we additionally pick two sequences $\bar{w}, \bar{w'} \in \VarSeq(M)$ such that $\bar{w} \initial \bar{w'}$ ($w$ is an initial segment of $w'$) and the differences $|\bar{w}| - b$, $|\bar{w'}| - |\bar{w}|$ are nonstandard, where $|\bar{w}|$, $|\bar{w'}|$ are the lengths of the respective sequences. We add the following axioms to $\Theta$:
\begin{itemize}
	\item  $S'(\eta_{a,\bar{w}}, \emptyset).$
	\item  $\neg S' (\eta_{a, \bar{w'}}, \emptyset)$.
\end{itemize} 
This guarantees that the quantifier correctness fails for $\eta_a$. The requirement on the length of $\bar{w}$ is to guarantee that the syntactic depth of $\eta_{a,\bar{w}}$ is high enough that the formula is not in the domain of $T$. 
	
If we take any model $(N',P',S') \models \Theta$ and restrict the predicate $S'$ to the formulae $\phi \in \form_{\LPA}(N')$ whose syntactic templates $\widehat{\phi}$ are in $\form_{\LPA}(N)$ (i.e., they essentially come from the original model), we obtain a model in $\mathscr{C}$ (compositionality and "non-failure" of quantifier correctness are easy to verify; the closure conditions on the domain follow automatically by elementarity and regularity). Therefore, it is enough to check that $\Theta$ is consistent which we will verify by a compactness argument. Let $\Theta_0 \subset \Theta$ be a finite subtheory. We will find a subset $S' \subset N$ such that $(N,S')$ satisfies $\Theta_0$. 

We introduce another technical notion which will be useful in this proof: if $\phi \in \form_{\LPA}(N)$ begins with a quantifier $Q$, then by its \df{root}, we mean  the unique formula $\psi$ such that $\phi = Q \bar{v} \psi$ for some $\bar{v} \in \VarSeq(M)$ and one of the following holds:
\begin{itemize}
	\item $\psi$ has syntactic depth $b$. 
	\item There is no formula satisfying the first condition and $\psi$ does not begin with a quantifier $Q$. 
\end{itemize}

Let $\Gamma_0$ be the (finite) set of formulae that appear under the satisfaction predicate $S'$ in the theory $\Theta_0$. Let $\phi_1, \ldots, \phi_n$ be an enumeration of formulae from $\Gamma_0$ and their roots (notice that we do not \emph{close} $\Gamma_0$ under taking roots; we apply the operation only once). Without loss of generality, we can assume that $\eta_{a,\bar{w}}, \eta_{a, \bar{w'}}$ are among $\phi_i$. (This will play any role only in the "initial step.") Consider the equivalence classes $[\phi_i]_{\sim}$ under the syntactic similarity relation $\sim$. We will construct $S'$ by induction on the following relation $\lhd$: $[\phi] \lhd [\psi]$ iff there exist $\phi' \in [\phi], \psi' \in \psi$ such that at least one of the following conditions holds:
\begin{itemize}
	\item $\phi'$ is a direct subformula of $\psi'$.
	\item $\phi'$ is unproblematic, there exists a sequence $\bar{v} \in \VarSeq(N)$ and $Q \in \{\forall, \exists\}$ such that $\psi' = Q\bar{v} \phi'$, and there is no $\bar{w} \initialeq \bar{v}$ such that $\psi' = Q \bar{w} \eta$, where $[\eta] = [\phi_i]$ for some $i \leq n$.
\end{itemize} 
The second condition is somewhat lengthy, but the idea is simple: we are allowed to make jumps over nonstandard blocks of quantifiers, but the relation $\lhd$ only holds for formulae linked by the shortest such jump. We could lift that restriction, but this would make some later parts of the proof awkward to express.

 We construct $S'$ as a union of a sequence of sets $S_j$ which we define by induction on the rank in the relation $\lhd$. The rank of a class $[\phi]$ is defined in a familiar way by the following equality:
\begin{displaymath}
\rk([\phi]) = \sup \set{\rk([\phi_i])}{i \leq n \wedge [\phi_i] \lhd [\phi] }.
\end{displaymath}
When showing the unboundedness of $\mathscr{C}$, we set:
\begin{displaymath}
S_0 = S.
\end{displaymath}
In this step, we are implicitly dealing with both the formulae from the domain of $S$ and the formulae $\phi$ whose class $[\phi]$ is $\lhd$-minimal. We implicitly make every such formula false under all assignments, since for such $\phi$ there is no $\alpha \in \Asn(\phi)$ such that $(\phi,\alpha) \in S_0$ (and it will remain so throughout the whole construction). Without loss of generality, we assume that $S$ contains all atomic formulae, i.e., equalities of terms, so they are also covered in this step.

When dealing with $(N,P,S) = (M,T,T)$, we instead define $S_0$ as the set of $(\phi,\alpha)$ such that one of the following conditions holds:
\begin{itemize}
	\item $(\phi,\alpha) \in T$.
	\item There exists a sequence $\bar{v} \in \VarSeq$ of a standard length $k \in \omega$ (possibly $k=0$) such that $\eta_{a,\bar{w}} = \exists \bar{v} \phi$ (and $\alpha \in \Asn(\phi)$ is arbitrary).
\end{itemize}
The second condition guarantees that the sentence $\eta_{a,\bar{w}}$ which we picked when we were defining the theory $\Theta$ will indeed be rendered true (by construction, the sentence $\eta_{a,\bar{w'}}$ and its direct subformulae will be rendered false).

We extend $S_j$ to $S_{j+1}$ so that the compositional conditions and quantifier correctness hold. More specifically, we define $S_{j+1}$ as a union of $S_j$ and the $(\phi,\alpha)$ such that $\rk([\phi]) = j+1$ and one of the following conditions holds:
\begin{itemize}
	\item For some $\psi \in \form_{\LPA}(N)$, $\phi = \neg \psi$ and $(\psi,\alpha) \notin S_{j}$.
	\item For some $\psi,\eta \in \form_{\LPA}(N)$, $\phi = \psi \vee \eta$ and $(\psi,\alpha) \in S_j$ or $(\eta,\alpha) \in S_j$.
	\item For some $\psi,\eta \in \form_{\LPA}(N)$, $\phi = \psi \wedge \eta$ and both $(\psi,\alpha) \in S_j$ and $(\eta,\alpha) \in S_j$.
	\item For some $\psi \in \form_{\LPA}(N), v \in \Var$, $\phi = \exists v \psi$ and there exists $\beta \sim_{v} \alpha$ such that $(\psi,\beta) \in S_j$. 
	\item For some $\psi \in \form_{\LPA}(N), v \in \Var$, $\phi = \forall v \psi$ and for all $\beta \sim_{v} \alpha$, $(\psi,\beta) \in S_j$.
	\item There exists $\bar{v} \in \VarSeq(N)$ and an unproblematic $\psi$ such that $\phi = \exists \bar{v} \psi$,  $[\psi] = [\phi_i]$ for some $i$, $\rk([\psi]) = j$, and there exists $\beta \sim_{\bar{v}} \alpha$ such that $(\psi,\beta) \in S_j$.  
	\item There exists $\bar{v} \in \VarSeq(N)$ and an unproblematic $\psi$ such that $\phi = \forall \bar{v} \psi$,  $[\psi] = [\phi_i]$ for some $i$, $\rk([\psi]) = j$, and for all $\beta \sim_{\bar{v}} \alpha$, $(\psi,\beta) \in S_j$. 
\end{itemize}

Since there are only finitely many classes $[\phi_i]$, there exists $j$ such that $S_j = S_{j+1}$. We define $S'$ as this last set in our construction. We now have to check that the constructed predicate satisfies the finite subtheory $\Theta_0 \subset \Theta$. 

Obviously, $(N,P,S')$ sastisfies the elementary diagram of $(N,P)$, $S'$ satisfies all the instances of the compatibility scheme and it extends $P$. 

Let us now check that if we consider the case where $(N,P,S) = (M,T,T)$, then indeed $\eta_{a,\bar{w}}$ is rendered true and $\eta_{a, \bar{w'}}$ is false. For either of the formulae, this is clear if its class is $\lhd$-minimal. If $\rk [\eta_{a,\bar{w}}] = j >0$, then since $\eta_{a,\bar{w}}$ is problematic, this means that the classes $[\eta_{a,\bar{u_1}}], \ldots [\eta_{a,\bar{u_j}}]$ are among $[\phi_i]$, where $\eta_{a,\bar{u_k}}$ is the formula obtained from $\eta_{a,\bar{w}}$ by removing the outermost $k$ quantifiers. Moreover, $[\eta_{a,\bar{u_j}}]$ is a minimal class. By construction, the formula $\eta_{a,\bar{u_j}}$ is satisfied under all assignments. Then using compositional conditions, we can check by induction that all the other formulae $\eta_{a,\bar{u_j}}$ are satisfied under all assignments. The proof for $\eta_{a,\bar{w'}}$ is analogous, but now we consider the formulae $\eta_{a,\bar{u'_k}}$ obtained by removing the outermost $k$ quantifiers from $\eta_{a, \bar{w'}}$. By construction, the minimal formula $\eta{a,\bar{u'_j}}$ is not satisfied under any assignment and we check that the same holds for all $\eta_{a,\bar{u'_k}}$ by induction using the compositional conditions.  

The set $S'$ satisfies all the instances of the compositionality scheme from $\Theta_0$ by construction and the compositionality of $S'$. Note that there is no conflict between defining $S_0$ to simply contain $S$ and further extending it by compositional conditions, since compositional clauses determine uniquely the behaviour of $S$ on a given formula given the behaviour on its subformulae. 

The quantifier correctness scheme is handled similarly. We have to check that there is no conflict between the definition of $S_0$ and how it was extended in the further steps. If $Q \bar{v} \phi$ is in the domain of $S$ and $\phi$ is unproblematic, then the root of $Q \bar{v} \phi$ is a subformula of $\phi$ (possibly not proper) and by assumption it is in the domain of $S$. Therefore, we can check by induction on rank that $S'$ is defined both on $\phi$ and $Q \bar{v} \phi$ by applying the quantifier correctness clause to a formula in the domain of $S$. (The assumption that $T$ is defined exactly on formulae of depth $\leq b$ guarantees that the definition of the root makes sense).

Regularity holds for $S_0$ by the assumption on $S$ in the proof that $\mathscr{C}$ is unbounded. In the case where $(N,P) = (M,T)$, we use the assumption that $T$ is regular and the fact that the similarity classes of formulae $\eta_{a,\bar{v}}$ are singletons (by definition of these specific formulae). Then we show by induction that the regularity holds for all other sets $S_j$. First, observe that for every class $[\phi]$ or rank $j+1$, the predicate $S_{j+1}$ is defined using the same condition  for every formula $\phi' \in [\phi]$, which follows by definition of syntactic similarity (all similar formulae have the same syntactic tree) and the fact that either all formulae in $[\phi]$ are problematic or none is (which, in turn, depends on the definition of $\eta_a$ and the assumption that the domain of  $P$ includes exactly the formulae of depth $b$). Then the claim follows, since applying the same compositional or quantifier correctness clause to all formulae in a class $[\phi]$ preserves regularity. This concludes the proof.
\end{proof}

The proof of Lemma \ref{lem_quantifier_correcntess_EV} will still be valid if we ignore the mention of the formulae $\eta_{\bar{u}}$ and require that the quantifier correctness holds for every formula.
\begin{corollary} 
Let $M \models \PA$ be countable and recursively saturated. Then there exists a regular full satisfaction class $S \subset M^2$ such that every formula $\phi \in \form_{\LPA}(M)$ satisfies the quantifier correctness.
\end{corollary}

We can formulate quantifier correctness as a pair of axioms for the compositional truth in a natural way as follows:
\begin{displaymath}
\forall \phi \in \form_{\LPA} \forall \tuple{v_1, \ldots, v_c} \in \VarSeq \ \Big(T \exists v_1 \ldots \exists v_c \phi \equiv \exists \alpha \in \Asn(\phi) T\phi[\alpha] \Big).
\end{displaymath}
\begin{displaymath}
\forall \phi \in \form_{\LPA} \forall \tuple{v_1, \ldots, v_c} \in \VarSeq \ \Big(T \forall v_1 \ldots \forall v_c \phi \equiv \forall \alpha \in \Asn(\phi) T\phi[\alpha] \Big).
\end{displaymath}
Then by the previous observation, we conclude that the following holds:
\begin{corollary} \label{cor_quantifier_correctness_conservative}
The theory $\CT^- + $ "the quantifier correctness holds for all formulae"  is conservative over $\PA$. 
\end{corollary}
The above two corollaries were essentially proved by Enayat and Visser as Theorem 6.1 in \cite{enayat_visser_long}.\footnote{Note that the formulation in \cite{enayat_visser_long} includes a claim that we can elementarily extend an arbitrary model $M \models \PA$ to a one with a satisfaction class which is disjunctively correct. Now we know that this part of the theorem is false. The crucial gap is that if $\phi \in M$ is a disjunction over a nonstandard set of formulae, it may have new disjuncts when considered in an elementary extension of $M$. The quantifier correctness part (called by the authors "existential correctness") is essentially proved with the same argument as the one presented here. However, the reader should be aware that certain, small and harmless, gaps can be found in the original presentation which originate from ignoring the possibility that quantifier blocks from a model $M$ may have new prefixes when considered in its elementary extension.} 

\section*{Acknowledgements}
This research was supported by an NCN MAESTRO grant 2019/34/A/HS1/00399 "Epistemic and Semantic Commitments of Foundational Theories."

\end{document}